\documentclass[notitlepage,twoside,a4paper]{amsart}
\usepackage{mathrsfs}
\usepackage{amsmath,amssymb,enumerate}
\usepackage{epsfig,fancyhdr,color}
\usepackage{epstopdf}
\usepackage{amssymb}
\usepackage{amsmath,amsthm}
\usepackage{latexsym}
\usepackage{amscd}
\usepackage{psfrag}
\usepackage{graphicx}
\usepackage{epsf}
\usepackage[latin1]{inputenc}
\usepackage[all]{xy}
\usepackage{tikz}
\usepackage{prettyref}
\usepackage{subfigure}
\usepackage{float}
\usepackage{color}
\usepackage{array}
\usepackage{cite}
\usepackage[totalwidth=15cm,totalheight=24cm]{geometry}
\newcommand{\II}{I\hspace{-0.1cm}I}
\newcommand{\III}{I\hspace{-0.1cm}I\hspace{-0.1cm}I}

\newtheorem{theorem}{\rm\bf Theorem}[section]
\newtheorem{proposition}[theorem]{\rm\bf Proposition}

\newtheorem{lemma}[theorem]{\rm\bf Lemma}
\newtheorem{corollary}[theorem]{\rm\bf Corollary}
\newtheorem{definition}[theorem]{\rm\bf Definition}
\newtheorem{remark}[theorem]{\rm\bf Remark}

\newtheorem{question}[theorem]{\rm\bf Question}
\newtheorem{claim}[theorem]{\rm\bf Claim}

\newtheoremstyle{named}{}{}{\itshape}{}{\bfseries}{.}{.5em}{#1 \thmnote{#3}}
\theoremstyle{named}

\newcommand{\HH}{{\mathbb H}}
\newcommand{\R}{{\mathbb R}}
\newcommand{\RP}{{\mathbb {RP}}}

\newcommand{\cL}{{\mathcal L}}

\newcommand{\cT}{{\mathcal T}}

\newcommand{\AdS}{\mathbb{A}\rm{d}\mathbb{S}}

\newcommand{\Isom}{\rm{Isom}}
\newcommand{\PSL}{\rm{PSL}}

\newcommand{\tr}{{\rm tr}}

\newcommand{\QF}{\mathcal{QF}_{AdS}(S)}
\newcommand{\cV}{\mathcal{V}}
\newcommand{\cI}{\mathcal{I}}

\newcounter{notes}%

\def\interieur#1{\mathord{\mathop{\kern 0pt #1}\limits^\circ}}

\title[]{Convex surfaces with prescribed induced metrics in anti-de Sitter spacetimes}

\author{Qiyu Chen}
\address{Qiyu Chen:
School of Mathematics, South China University of Technology, 510641, Guangzhou, P. R. China.}
\email{qiyuchen@scut.edu.cn}
\thanks{Q. Chen was partially supported by NSFC No: 12101244, Guangdong Basic and Applied Basic Research Foundation No: 2022A1515012225, and Guangzhou Science and Technology Program No: 202201010464.}

\author{Jean-Marc Schlenker}
\address{Jean-Marc Schlenker:
University of Luxembourg, FSTM, Department of Mathematics, Maison du nombre, 6 avenue de la Fonte,
L-4364 Esch-sur-Alzette, Luxembourg}
\email{jean-marc.schlenker@uni.lu}
\thanks{J.-M. S. was partially supported by FNR project O20/14766753. J.-M.S. would like to thank IHES and MPIM Bonn, where part of this work was completed.}

\date{v0, \today}

\begin{document}

\begin{abstract}
  Let $S$ be a closed surface of genus at least $2$, let $h$ be a smooth metric of curvature $K<-1$ on $S$, and let $h_0$ be a hyperbolic metric on $S$. We show that there exists a unique quasifuchsian AdS spacetime with left metric isotopic to $h_0$, containing a past-convex Cauchy surface with induced metric isotopic to $h$.
\end{abstract}

\maketitle

\tableofcontents

\section{Introduction}

\subsection{Main result}

We consider here 3-dimensional maximal globally hyperbolic and spatially compact (MGHC) anti-de Sitter (AdS) spacetimes. If $M$ is a MGHC AdS spacetime, it contains closed Cauchy surfaces, which are all homeomorphic to a fixed closed surface $S$ which we will assume to be of genus at least $2$.

Those spacetimes are known since the work of Mess \cite{mess} to have many remarkable similarities with quasifuchsian hyperbolic manifolds, and are for this reason often called ``quasifuchsian AdS spacetimes'', a terminology we will follow here. The identity component of the isometry group of the 3-dimensional AdS space $\AdS^3$ identifies (up to finite index) with ${\rm PSL}(2,\R)\times {\rm PSL}(2,\R)$, and, if $M$ is a quasifuchsian AdS spacetime, its holonomy representation $\rho:\pi_1M=\pi_1(S)\to {\rm PSL}(2,\R)\times {\rm PSL}(2,\R)$ is the product $\rho=(\rho_L,\rho_R)$ of two Fuchsian representations, called its left and right representations. Those representations $\rho_L$ and $\rho_R$ are therefore the holonomy representations of two hyperbolic metrics $h_L$ and $h_R$ on $S$, called the left and right metrics of $M$, see \cite{mess,mess-notes}.

We denote by $\cT_S$ the Teichm\"uller space of $S$, which is the space of hyperbolic metrics on $S$ (considered up to isotopy) and denote by $\mathcal{QF}_{AdS}(S)$ the deformation space of quasifuchsian AdS spacetimes with Cauchy surfaces homeomorphic to $S$.

The main result here is the following statement, which  gives an affirmative answer to Question 6.11 in \cite{weylsurvey}.

\begin{theorem}\label{thm:closed-surf}
Let $h$ be a complete Riemannian metric of curvature $K<-1$ on $S$, and let $h_0\in\cT_S$ be a hyperbolic metric on $S$. There is a unique quasifuchsian AdS spacetime $M\in\QF$ with left metric isotopic to $h_0$ and containing a past-convex spacelike surface with induced metric isotopic to $h$.
\end{theorem}

Tamburelli \cite[Proposition 7.1]{tamburelli2016} showed that the existence holds, based on a topological intersection argument. To show the uniqueness, the key point is to prove the infinitesimal rigidity of quasifuchsian AdS metrics with respect to the left metric and the induced metric on a past-convex Cauchy surface (see Proposition \ref{prop:local-rigidity}).

When $h$ has constant curvature, Theorem \ref{thm:closed-surf} follows from the ``Landslide Theorem'' in \cite[Theorem 1.15]{cyclic}. Indeed, for each constant $K<-1$, a quasifuchsian AdS spacetime $M$ contains a unique past-convex Cauchy surface $S_K$ with induced metric of constant curvature $K$ (see \cite{BBZ2}), and the left (resp. right) hyperbolic metric of $M$ is obtained by a landslide of parameter $\theta$ (resp. $-\theta$) with $\theta\in (0,\pi)$ and $K=-1/\cos^2(\theta/2)$, determined by the induced metric and third fundamental form on $S_K$, see \cite[Lemma 1.9]{cyclic}. Since there is a unique landslide map with parameter $\theta\in(0,2\pi)$ (resp. $\theta\in (-2\pi,0)$) taking a given hyperbolic metric to another one (see \cite[Theorem 1.15]{cyclic}), there is a unique choice of $M$ such that the induced metric on $S_K$ and the left (resp. right) hyperbolic metric are prescribed. This argument actually extends to the case where $h$ has constant curvature $-1$, with Thurston's Earthquake Theorem \cite{thurston:earthquakes,mess} used instead of the result on landslide.

It follows from Theorem \ref{thm:closed-surf}, using an elementary symmetry argument, that there also exists a unique quasifuchsian AdS spacetime having $h_0$ as right hyperbolic metric, and containing a past-convex surface with induced metric $h$. Another direct symmetry argument shows that one can prescribe the induced metric on a future-convex, rather than past-convex, spacelike surface.

There is a well-defined duality between locally strongly convex spacelike surfaces in quasifuchsian AdS spacetimes (see \cite[Section 11.1]{BBZ2} or \cite[Section 2.6]{cyclic2}), where by ``strongly convex'' we mean that the second fundamental form is everywhere positive definite. This duality associates to every past-convex spacelike surface $\Sigma$ a surface $\Sigma^*$, defined as the space of points dual to the oriented tangent planes of $\Sigma$. This construction defines a duality map from $\Sigma$ to $\Sigma^*$. If $\Sigma$ is spacelike and strongly past-convex, then $\Sigma^*$ is spacelike and strongly future-convex, and the pull-back by the duality map of the induced metric on one is the third fundamental form on the other. Moreover, the curvature $K^*$ of the induced metric on $\Sigma^*$ at the point $x^*$ which is the image under the duality map of a point $x\in \Sigma$ is equal to $K^*=-K/(1+K)$, where $K$ is the curvature of the induced metric on $\Sigma$ at $x$.
Through this duality, Theorem \ref{thm:closed-surf} also implies that given a hyperbolic metric $h_0$ and another metric $h$ on $S$ with curvature $K<-1$, there is a unique quasifuchsian AdS spacetime with left (or right) hyperbolic metric $h_0$, and containing a future-convex (or past-convex) spacelike surface with third fundamental form $h$.

\subsection{Background and motivations}
\label{subsec:background}

Our main motivation comes from hyperbolic geometry. As already mentioned, quasifuchsian (that is, %GHMC)
MGHC) AdS spacetimes have deep analogies with quasifuchsian hyperbolic manifolds.  There are several questions one can consider in a quasifuchsian hyperbolic manifold, for instance whether one can find a (unique) quasifuchsian hyperbolic manifold containing a (unique) geodesically convex subset such that the induced metric (or third fundamental form) on the boundary is prescribed, see \cite{weylsurvey}. Here we are motivated by a slightly different type of question, in the setting of {\em hyperbolic ends}. A prime example of hyperbolic end is a connected component of the complement of the convex core in a convex co-compact hyperbolic manifold, but the following definition is somewhat more general.

\begin{definition} \label{df:hypend}
  A {\em hyperbolic end} is a non-complete 3-dimensional hyperbolic manifold $(E,g)$, where $E=S\times (0,\infty)$, complete on the side corresponding to $\infty$, and having as metric completion a pleated concave surface at the side corresponding to $0$.

  Given a hyperbolic end $E$, we denote by $\partial_0E$ the pleated surface boundary component corresponding to $0$, and by $\partial_\infty E$ the asymptotic boundary component corresponding to $\infty$.
\end{definition}

The asymptotic boundary $\partial_\infty E$ of $E$ is naturally equipped with a conformal structure, and more precisely with a complex projective structure. The pleated boundary $\partial_0E$ is equpped with a hyperbolic metric and with a measured geodesic lamination measuring its ``pleating''.

Such a hyperbolic end contains many embedded surfaces $S\subset E$ such that $\partial_0E$ and $S$ bound together a relatively compact geodesically convex subset $\Omega$ of $E$ -- we call such a surface a {\em locally convex surface} in $E$.

Given such a smooth, locally convex surface $S\subset E$, the hyperbolic Gauss map -- which sends a point $x\in S$ to the endpoint at infinity of the geodesic ray starting from $x$, normal to $S$, towards the concave side -- is a homeomorphism from $S$ to $\partial_\infty E$.

We are interested in the following question and its dual, see also \cite{weylsurvey}.

\begin{question} \label{q:end}
  Let $S$ be a closed surface of genus at least $2$, let $h$ be a smooth metric of curvature $K>-1$ on $S$, and let $c\in \cT_S$ be a conformal structure on $S$. Is there a unique hyperbolic end $E$ containing a locally convex surface $\Sigma$ homeomorphic to $S$, with induced metric isotopic to $h$ and such that the pull-back by the hyperbolic Gauss map of the conformal structure at infinity of $E$ is isotopic to $c$?
\end{question}

\begin{question} \label{q:end*}
  Let $S$ be a closed surface of genus at least $2$, let $h^*$ be a smooth metric of curvature $K^*<1$ on $S$ such that all closed, contractible geodesics of $(S,h^*)$ have length strictly larger than $2\pi$, and let $c\in \cT_S$ be a conformal structure on $S$. Is there a unique hyperbolic end $E$ containing a locally convex surface $\Sigma$ homeomorphic to $S$, with third fundamental form isotopic to $h^*$ and such that the pull-back by the hyperbolic Gauss map of the conformal structure at infinity of $E$ is isotopic to $c$?
\end{question}

The answer to Question \ref{q:end} is positive when $h$ has constant curvature $-1$, this corresponds to the main result in \cite{scannell-wolf}. When $h$ (resp. $h^*$) has constant curvature $K\in (-1,0)$ (resp. $K^*\in (-\infty,0)$), Questions \ref{q:end} and \ref{q:end*} correspond to Questions 9.1 and 9.2 in \cite{cyclic}. The existence part is proved for both questions in Theorem 1.9 and Theorem 1.10 of \cite{cyclic2}.

Question \ref{q:end*} also has a positive answer in the ``limit'' case where $h^*$ is replaced by a measured lamination, this is the main result of \cite{dumas-wolf}. In another limit case where $h^*$ corresponds to the pleating pattern associated to a circle packing or circle pattern, Question \ref{q:end*} corresponds to a well-known conjecture of Kojima, Mizushima and Tan, see \cite{KMT,KMT2,KMT3,delaunay}.

In the analogy discovered by Mess \cite{mess,mess-notes}, the conformal metrics at infinity of a quasifuchsian hyperbolic manifold correspond to the left and right hyperbolic metrics of a quasifuchsian (or %GHMC)
MGHC) AdS spacetime. In this sense, Theorem \ref{thm:closed-surf} provides a positive answer to the AdS analog of Question \ref{q:end}, and also, through the duality %in AdS
between strongly convex spacelike surfaces in quasifuchsian AdS spacetimes mentioned above, to the AdS analog of Question \ref{q:end*}.

\subsection{Outline of the paper}

Section 2 contains background definition and results on AdS geometry, on quasifuchsian AdS spacetimes, as well as on geometric data on immersed surfaces in AdS spacetimes. The proof of the main result is in Section 3, while Section 4 singles out an application.

\section{Preliminaries}

\subsection{The 3-dimensional anti-de Sitter geometry}
Let $\R^{2,2}$ denote the real 4-dimensional vector space $\R^4$ equipped with a symmetric bilinear form $\langle \cdot, \cdot\rangle_{2,2}$ of signature $(2,2)$, where $\langle x, y\rangle_{2,2}:=x_1y_1+x_2y_2-x_3y_3-x_4y_4$. The \textit{quadric model} of the 3-dimensional \textit{anti-de Sitter} (AdS) space is
$$AdS^3:=\{x\in\R^{2,2} ~|~ \langle x,x\rangle_{2,2}=-1\}~,$$
equipped with a pseudo-Riemannian metric of signature $(2,1)$, which is induced by restricting the bilinear form $\langle \cdot, \cdot\rangle_{2,2}$ to the tangent space at each point of of $AdS^3$. It is a 3-dimensional Lorentzian symmetric space of constant curvature $-1$ diffeomorphic to $\HH^2\times \mathbb{S}^1$.

    The \emph{projective model} $\AdS^3$ of 3-dimensional AdS space is defined as the image of $AdS^3$ under the projection $\pi:\mathbb{R}^{2,2}\backslash\{0\}\rightarrow \mathbb{RP}^{3}$, with the metric induced from the projection. It is clear that the quadric model $AdS^3$ is a double cover of the projective model $\AdS^3$.

    The boundary of $\AdS^3$, denoted by $\partial \AdS^3$, is the image of  $\{x\in\mathbb{R}^{2,2}:\langle x,x\rangle_{2,2}=0\}$ under $\pi$, which is foliated by two families of projective lines, called the {\em left} and {\em right leaves}, respectively.

    The geodesics in $\AdS^3$ are given by the intersection of projective lines in $\RP^3$ with $\AdS^3$: the {\em spacelike} geodesics correspond to the projective lines intersecting the boundary $\partial\AdS^3$ in two points, while {\em lightlike} geodesics are tangent to $\partial\AdS^3$, and {\em timelike} geodesics are disjoint from $\partial\AdS^3$.

   The isometry group of $AdS^3$ is the indefinite orthogonal group ${\rm O}(2,2)$ of reversible linear transformations on $\R^{2,2}$ preserving the bilinear form $\langle \cdot, \cdot \rangle_{2,2}$. The group of orientation and time-orientation preserving isometries of $\AdS^3$, denoted by $\Isom_0(\AdS^3)$,  is the identity component of ${\rm PO}(2,2)={\rm O}(2,2)/\{\pm {\rm Id}\}$, which can be identified with $\PSL(2,\R)\times\PSL(2,\R)$. One simple way to see this is because the boundary $\partial\AdS^3$ is foliated by two families of lines, and those two foliations are invariant under any element of $\Isom_0(\AdS^3)$, since isometries act projectively in the projective model. The set of leaves of each foliation is equipped with a real projective structure by the intersections with any leave of the other foliation, and the action of $\rm{Isom}_0(\AdS^3)$ defines in this way two elements of $\PSL(2,\R)$.

    We always use the projective model $\AdS^3$ throughout the paper. It is interesting to notice that $\AdS^3$ has a Lie group model, in which $\AdS^3$ is identified with $\PSL(2,\R)$. More details about the geometry of anti-de Sitter space can be found in e.g. \cite{bonsante-seppi:anti,mess,mess-notes,adsquestions}.

\subsection{Quasifuchsian AdS spacetimes}
An AdS spacetime is a Lorentzian 3-manifold locally isometric to $\AdS^3$. An AdS spacetime $M$ is {\em  maximal globally hyperbolic and spatially compact} (MGHC) if
 \begin{itemize}
 \item $M$ is globally hyperbolic and spatially compact (GHC): $M$ contains a closed Cauchy surface (i.e. a spacelike surface intersecting each inextensible timelike curve exactly once).
 \item any isometric embedding of $M$ into a GHC AdS spacetime is an isometry.
 \end{itemize}

 It is well-known that a globally hyperbolic spacetime is topologically a product of any of its Cauchy surfaces with an interval, see e.g. \cite[Chapter 3]{Beem1996}.
 %%\marginnote{reference: John K. Beem, Paul E. Ehrlich, and Kevin L. Easley, Global Lorentzian geometry, 2nd ed., Monographs and Textbooks in Pure and Applied Mathematics, vol. 202, Marcel Dekker, Inc., New York, 1996.}
MGHC AdS spacetimes have been shown by G. Mess \cite{mess,mess-notes} to present remarkable analogies with quasifuchsian hyperbolic manifolds, and they are now often called {\em quasifuchsian AdS spacetimes} in the mathematics literature.

It was proved by Mess \cite{mess} that given any MGHC AdS spacetime with Cauchy surfaces homeomorphic to $S$, its holonomy representation $\rho:\pi_1(S) \rightarrow \Isom_0(\AdS^3)=\PSL(2,\R)\times\PSL(2,\R)$ can be written as $\rho=(\rho_L, \rho_R)$, where $\rho_L, \rho_R: \pi_1(S)\rightarrow {\rm PSL}(2,\R)$ have maximal Euler number, and are therefore (by a result of Goldman \cite{Goldman}) {\em Fuchsian} representations (i.e. the holonomy representations of hyperbolic metrics on $S$). So the space $\QF$ appearing in Section \ref{subsec:background} is exactly the deformation space of MGHC AdS spacetimes with Cauchy surfaces homeomorphic to $S$.

A quasifuchsian AdS spacetime $M\in\QF$ contains a non-empty closed subset $N$ which is {\em geodesically convex} (i.e. each geodesic segment of $M$ connecting any two points of $N$ is contained in $N$). In particular, the smallest non-empty geodesically convex closed subset of $M$ is called the {\em convex core} of $M$, denoted by $C_M$. The boundary of $C_M$ is the union of two (possibly identified) spacelike surfaces. In the Fuchsian case (i.e. the left and right metrics are equal), $C_M$ degenerates to a totally geodesic spacelike surface. In the non-Fuchsian case,  each boundary component of $C_M$ is a spacelike surface pleated along a measured geodesic lamination, whose measures record the bending angles along totally geodesic pieces. In both cases, the induced metrics on the two boundary components of $C_M$ are hyperbolic.

 \subsection{The linearized Gauss and Codazzi equations}

Let $(M,g)$ be a quasifuchsian AdS spacetime and let $\Sigma\subset M$ be an (embedded) spacelike surface with induced metric $I$. The {\em shape operator} of $\Sigma$ is the bundle morphism $B:T\Sigma\rightarrow T\Sigma$ defined by
    \begin{equation*}
        B(u)=-\nabla_{u}n~,
    \end{equation*}
    where $n$ is the future-directed unit normal vector field on $\Sigma$, $u$ is a tangent vector field to $\Sigma$ and $\nabla$ is the Levi-Civita connection of $g$. The {\em second and third fundamental forms} of $\Sigma$ are respectively defined by
    \begin{equation*}
        \II(u,v)=I(Bu,v)~,\qquad
        \III(u,v)=I(Bu,Bv)~,
    \end{equation*}
    for all $u,v\in T_p\Sigma$ with $p\in \Sigma$. We call the pair $(I,B)$ of induced metric and shape operator of $\Sigma$ the {\em embedding data} of $\Sigma$.

    It is well-known that $(I,B)$ satisfies the following {\em Gauss and Codazzi equations} respectively:
\begin{equation*}\label{eq:Gauss}
\begin{split}
-1-\det(B)&=K~,\\
d^{D}B&=0~,
\end{split}
\end{equation*}
where $K$ is the Gaussian curvature of $I$,  $D$ is the Levi-Civita connection of $I$, while the exterior derivative $d^D$ of $B$ (considered as a one-form with values in $T\Sigma$) is defined as
\begin{equation*}
(d^DB)(u,v)=D_u(Bv)-D_v(Bu)-B([u,v])~,
\end{equation*}
for any two tangent vector fields $u$, $v$ on $\Sigma$.

Take the first-order derivatives of both sides of the Gauss equation and assume that $B$ is non-degenerate. A direct computation shows that the shape operator $B$ satisfies the following
\begin{equation*}\label{eq:linear-Gauss}
\tr(B^{-1}\dot B)=0~,
\end{equation*}
called the {\em linearized Gauss equation}.

A spacelike surface $\Sigma$ in $M$ is said to be {\em strongly convex} if the determinant of the shape operator $B$ of $\Sigma$ is positive definite. In particular, $\Sigma$ is said to be {\em strongly future-convex} (resp. {\em strongly past-convex}) if the principal curvatures of $B$ are both negative (resp. positive).

\subsection{The left and right metrics on a spacelike surface}

 Let $\Sigma$ be an embedded (smooth) spacelike surface in a quasifuchsian AdS spacetime $M$, with the embedding data $(I,B)$. Let $J$ be the complex structure on $\Sigma$ defined by $I$. It was shown in \cite[Lemma 3.16]{minsurf} that the left and right metrics of $M$, denoted by $I^\#_+$ and $I^\#_-$, can be defined alternatively by
\begin{equation*}\label{def:left-metric}
 I^\#_{\pm} = I((E\pm JB)\cdot, (E\pm JB)\cdot)~,
\end{equation*}
in the sense that if $I^\#_\pm$ are defined in this manner as Riemannian metrics on $\Sigma$, they are isometric to $h_L$ and $h_R$, respectively.

Moreover, this definition is independent on the choice of embedded smooth spacelike surfaces. Since the left and right metrics play symmetric roles in our question, without loss of generality, we only consider the left metric of $M$ here.

\section{Proof of the main result}

Let $(M,g)\in\QF$ be a quasifuchsian AdS spacetime $M$ (with its metric denoted by $g$) which has left metric $h_L$ and contains an embedded smooth, strongly past-convex spacelike surface $\Sigma$ with induced metric $I$ (and with the metric $I^\#_+$ isotopic to $h_L$).

\subsection{Local rigidity}

In this section, we aim to show the infinitesimal rigidity of $g$ with respect to the left metric $I^\#_+$ and the induced metric $I$ on $\Sigma$, as stated below.

\begin{proposition}\label{prop:local-rigidity}
Any first order deformation of $g\in \QF$ which preserves the left metric $I^\#_+$ and the induced metric $I$ on the strongly past-convex spacelike surface $\Sigma$ at first order is trivial.
\end{proposition}

We first consider the first-order deformation of the embedding data $(I,B)$ of the strongly past-convex spacelike surface $\Sigma$.

\begin{lemma} \label{st:ads}
  Let $(I,B)$ be the induced metric and shape operator of a strongly past-convex spacelike surface $\Sigma$ in a quasifuchsian AdS spacetime. Consider a first-order deformation such that $I$ is fixed
  and $\dot B$ satisfies the Codazzi and linearized Gauss equations and that the corresponding variation of $I^\#_+$ is trivial, i.e. induced by a vector field say $v$ on $\Sigma$. Then $\dot B=0$.
\end{lemma}

Let $(I_t,B_t)_{t\in [0,\epsilon]}$ be a smooth one-parameter family of embedding data with $(I_0,B_0)=(I,B)$, such that the first-order variation $\dot I=I'_0$ of $I_t$ at $t=0$ vanishes.
Let $J_t$ denote the complex structure on $\Sigma$ defined by $I_t$. It follows directly that $J_0=J$ and $\dot{J}=0$. The first-order variation of the left metric $I^\#_+$ satisfies that
\begin{equation*}
\begin{split}
\dot I^\#_+&=\frac{d}{dt}|_{t=0}I_t((E+J_tB_t)\cdot,(E+J_tB_t)\cdot)\\
&=I(J\dot{B},(E+JB)\cdot)+I((E+JB)\cdot,J\dot{B}\cdot)\\
&=I^\#_+((E+JB)^{-1}J\dot{B}\cdot, \cdot)+I^\#_+(\cdot,(E+JB)^{-1}J\dot{B}\cdot)~.
\end{split}
\end{equation*}

\begin{definition}\label{def:b}
  We define a bundle morphism $b:T\Sigma\rightarrow T\Sigma$ as
  \begin{equation*}
  b=(E+JB)^{-1}J\dot B~.
  \end{equation*}
\end{definition}
Then the expression of $\dot{I}^\#_+$ can be rewritten as
\begin{equation} \label{eq:var-1}
\dot I^\#_+ = I^\#_+(b\cdot, \cdot)+I^\#_+(\cdot, b\cdot)~.
\end{equation}

\begin{lemma}\label{lm:Codazzi-Gauss}
 Let $b$ be the bundle morphism defined in Definition \ref{def:b}, then %% we have
  \begin{equation}
    \label{eq:b1}
    \tr((E+JB)b)=0~,
  \end{equation}
  \begin{equation}
    \label{eq:b2}
    \tr((E+(JB)^{-1})b)=0~.
  \end{equation}
\end{lemma}

\begin{proof}
  By definition, $b=(E+JB)^{-1}J\dot B$, so $J\dot B=(E+JB)b$.
  Note that $B$ is self-adjoint for $I$, so since $\dot B$ is a first-order variation of $B$ still satisfying this condition,
  $\dot B$ is also self-adjoint for $I$. Therefore $\tr(J\dot B)=0$, so $\tr((E+JB)b)=0$, which proves \eqref{eq:b1}.

  In addition, since $\Sigma$ is strongly convex, the shape operator $B$ of $\Sigma$ is non-degenerate. Using $b=(E+JB)^{-1}J\dot B$, the linearized Gauss equation $\tr(B^{-1}\dot B)=0$ then translates as
  $$ \tr(-B^{-1}J(E+JB)b)=0~. $$
  Since $(JB)^{-1}=-B^{-1}J$, this can be written as
  $$ \tr((E+(JB)^{-1})b)=0~, $$
  proving \eqref{eq:b2}.
\end{proof}

\begin{remark}\label{eq:JB-inverse}
  Since $\tr(JB)=0$, the Cayley-Hamilton theorem shows that $(JB)^2 + \det(JB)E=0$. Since $\det(JB)=\det(B)$ and using the Gauss equation $-1-\det(B)=K$ (where $K$ is the Gaussian curvature of $I$), it follows that
  $$JB=(1+K)(JB)^{-1}~.$$
\end{remark}

Using this remark, equations \eqref{eq:b1} and \eqref{eq:b2} can be simplified and we arrive at the following statement.

\begin{lemma}\label{lm:equiv-equa}
  The system of equations \eqref{eq:b1} and \eqref{eq:b2} is equivalent to the following
  \begin{equation}
    \label{eq:b4}
    \tr(b)=0~,
  \end{equation}
  \begin{equation}
    \label{eq:b5}
    \tr(JBb)=0~.
  \end{equation}
\end{lemma}

\begin{proof}
We first show that the system of equations \eqref{eq:b1} and \eqref{eq:b2} implies that of \eqref{eq:b4} and \eqref{eq:b5}. Since the surface $\Sigma$ is strongly convex, we have $K=-1-\det(B)<-1$, which ensures that $K+1<0$. Using Remark \ref{eq:JB-inverse}, equation \eqref{eq:b2} can be written as
\begin{equation}
  \label{eq:b2a}
  \tr((E+\frac{JB}{K+1})b)=0~.
\end{equation}
Taking a linear combination of \eqref{eq:b1} and of \eqref{eq:b2a} shows \eqref{eq:b4} and \eqref{eq:b5}.

The other direction for the equivalence is clear using Remark \ref{eq:JB-inverse}.
\end{proof}

For the convenience of computation, we introduce the following lemma, see e.g. \cite[Proposition 3.12]{minsurf}.

\begin{lemma}\label{prop:connection-curv}
  \label{computation of connection and curvature}
  Let $\Sigma$ be a surface with a Riemann metric $g$. Let $A:T\Sigma\rightarrow T\Sigma$ be a bundle morphism such that $A$ is everywhere invertible and $d^{\nabla}A=0$, where $\nabla$ is the Levi-Civita connection of $g$. Let $h$ be the symmetric (0,2)-tensor defined by $h=g(A\cdot,A\cdot)$. Then the Levi-Civita connection of $h$ is given by
  \begin{equation*}
    \nabla^{h}_{u}(v)=A^{-1}\nabla_{u}(Av)~,
  \end{equation*}
  and its curvature is given by
  \begin{equation*}
    K_{h}=\frac{K_{g}}{\det(A)}~.
  \end{equation*}
\end{lemma}

\begin{claim}\label{clm:Codazzi-b}
  The bundle morphism $b$ defined in Definition \ref{def:b} satisfies the Codazzi equation:
  \begin{equation*}
    d^{D^\#}b=0~,
  \end{equation*}
  where $D^\#$ is the Levi-Civita connection of $I^\#_+$.
\end{claim}

\begin{proof}
  By Lemma \ref{prop:connection-curv}, the Levi-Civita connection of $I^\#_+ = I((E+JB)\cdot, (E+JB)\cdot)$ is
  $$ D^\#_uv=(E+JB)^{-1}D_u((E+JB)v)~. $$

  Combined this with the definition of $d^{D^\#}b$, we obtain
  \begin{equation*}
    \begin{split}
      (d^{D^\#}b)(u,v) & =D^\#_u(bv)-D^\#_v(bu)-b([u,v])\\
                       &=(E+JB)^{-1}D_u((E+JB)bv)-(E+JB)^{-1}D_v((E+JB)bu)-b([u,v])\\
                       &=(E+JB)^{-1}D_u(J\dot{B}v)-(E+JB)^{-1}D_v(J\dot{B}u)-(E +JB)^{-1}J\dot{B}([u,v])\\
                       &=(E+JB)^{-1}\big(D_u(J\dot{B}v)-D_v(J\dot{B}u)-J\dot{B}([u,v])\big)\\
                       &=(E+JB)^{-1}J(d^D\dot{B})(u,v)\\
                       &=0~.
    \end{split}
  \end{equation*}
  The third equality uses the definition of $b$, while the last equality uses the fact that $B_t$ satisfies the Codazzi equation.
\end{proof}

\begin{claim}\label{clm:b-v}
Under the assumptions of Lemma \ref{st:ads}, there exists a function $\mu:\Sigma\to \R$ such that the bundle morphism $b$ defined in Definition \ref{def:b} satisfies
\begin{equation}\label{eq:b-assump}
b=D^\# v+\mu J^\#~,
\end{equation}
where $J^\#$ is the complex structure of $I^\#_+$.
\end{claim}

\begin{proof}
Under the assumptions of Lemma \ref{st:ads}, the variation of $I^\#_+$ is induced by a vector field $v$ on $\Sigma$, that is, $\dot I^\#_+=\cL_vI^\#_+$. A direct computation shows that
$$\cL_vI^\#_+ = I^\#_+(D^\#_\cdot v,\cdot) + I^\#_+(\cdot, D^\#_\cdot v)~.$$
Combining \eqref{eq:var-1}:
$\dot I^\# = I^\#_+(b\cdot, \cdot)+I^\#_+(\cdot, b\cdot)$, this means that $b$ and $D^\# v$ have the same self-adjoint component, or in other terms $b-D^\#v$ is anti-self-adjoint for $I^\#$. %% Note that $\Sigma$ is a surface so %% $b-D^\#v$ corresponds to an anti-symmetric two by two matrix in  a basis of $I^\#$. Let $J^\#$ be the complex structure on $\Sigma$ defined by $I^\#$. It follows that
So there exists a function $\mu:\Sigma\to \R$ such that
$$ b-D^\# v=\mu J^\#~. $$
The claim follows.
\end{proof}

The Codazzi equation applied to $b$ and the fact that $I^\#_+$ has constant curvature $-1$ then lead to the following lemma.

\begin{lemma}\label{lm:dD-computation}
  The vector field $v$ on $\Sigma$ and the function $\mu:\Sigma\to \R$ in Claim \ref{clm:b-v} satisfy
  $$v + J^\# D^\# \mu=0~.$$
\end{lemma}

\begin{proof}
  We compute $d^{D^\#}$ of each term of the right-hand side of \eqref{eq:b-assump}. For all $x,y\in T_p\Sigma$ with $p\in\Sigma$,
  \begin{eqnarray*}
    (d^{D^\#}D^\#v)(x,y) & = & D^\#_xD^\#_yv-D^\#_yD^\#_xv-D^\#_{[x,y]}v \\
                       & = & (-K^\#J^\# v) da^\#(x,y) \\
                       & = & (J^\# v) da^\#(x,y)~, \\
    (d^{D^\#}\mu J^\#)(x,y) & = & D^\#_x(\mu J^\#y) - D^\#_y(\mu J^\# x) - \mu J^\#[x,y] \\
                       & = & d\mu(x)J^\#y - d\mu(y)J^\#x~. \\
    & = & -(D^\#\mu) da^\#(x,y)~,
  \end{eqnarray*}
  where $K^\#$ is the Gaussian curvature of $I^{\#}_+$ and $da^\#$ is the area element of $I^\#_+$.

  Putting the two terms together and combining with Claim \ref{clm:Codazzi-b} and \eqref{eq:b-assump}, the result follows.
\end{proof}

Replacing this in the expression \eqref{eq:b-assump} of $b$ then leads to the following.

\begin{claim}\label{clm:b-mu}
   Under the assumptions of Lemma \ref{st:ads}, the bundle morphism $b$ defined in Definition \ref{def:b} has the following expression:
  \begin{equation}
    \label{eq:b-new}
    b = J^\# (-D^\# D^\# \mu + \mu E)~,
  \end{equation}
  where $\mu$ is the function given in Claim \ref{clm:b-v}.
\end{claim}

\begin{lemma}
Under the assumptions of Lemma \ref{st:ads}, we have
\begin{equation}\label{eq:b6}
   \tr(JBJ^\#(-D^\# D^\# \mu+\mu E))=0~.
\end{equation}
\end{lemma}

\begin{proof}
% We first show that the system of equations \eqref{eq:b4} and \eqref{eq:b5} is equivalent to \eqref{eq:b6}. We now note that \eqref{eq:b4} is empty, since $b$ is traceless by \eqref{eq:b-new}. So we only have one equation left:
% $$\tr(JBb)=0~.$$
% Using \eqref{eq:b-new} in Claim \ref{clm:b-mu}, this can be written as \eqref{eq:b6}. Combining Lemma \ref{lm:Codazzi-Gauss} and Lemma \ref{lm:equiv-equa}, the lemma follows.
This follows directly from Equation \eqref{eq:b5}, with the expression of $b$ in the previous claims.
\end{proof}

It can be noted that \eqref{eq:b-new} implies that $b$ is traceless, so \eqref{eq:b4} is always satisfied.

\begin{lemma}\label{lm:neg-def}
  The operator $JBJ^\#$ in \eqref{eq:b6} is self-adjoint for $I^\#_+$ and negative definite. More precisely, if $(e_1, e_2)$ is an orthonormal basis for $I$ of eigenvectors of $B$, with eigenvalues $k_1, k_2$ then $((E+JB)^{-1}(e_1),(E+JB)^{-1}(e_2))$ is an orthonormal basis of $I^\#_+$ of eigenvectors of $JBJ^\#$, with eigenvalues $-k_2, -k_1$.
\end{lemma}

\begin{proof}
   Recall that $J^\#=(E+JB)^{-1}J(E+JB)$ and $I^\#_+=I((E+JB)\cdot,(E+JB)\cdot)$. Since $B$ is self-adjoint for $I$, and $J$ is the complex structure of $I$, a direct computation shows that
    $$I^\#_+(JBJ^\#\cdot,\cdot)=I^\#_+(\cdot,JBJ^\#\cdot)~,$$
    which means that $JBJ^\#$ is self-adjoint for $I^\#_+$.

Let $E_i=(E+JB)^{-1}(e_i), i=1,2$. Then
  $$ J^\# E_i=(E+JB)^{-1} Je_i~. $$
  However, $JB$ commutes with $E+JB$ and therefore also with $(E+JB)^{-1}$. As a consequence,
  $$ JBJ^\# E_i=JB(E+JB)^{-1} Je_i=(E+JB)^{-1}JBJe_i~. $$
  Clearly, $JBJe_1=-k_2e_1$ while $JBJe_2=-k_1e_2$. Then
   $$ JBJ^\# E_1=-k_2(E+JB)^{-1}e_1=-k_2E_1~,$$
   and $$ JBJ^\# E_2=-k_1(E+JB)^{-1}e_2=-k_1E_2~.$$
  The result follows.
\end{proof}

\begin{remark}
Another (simpler) way to obtain the result is by checking that
$$ JBJ^\# = (E+JB)^{-1}JBJ(E+JB)~. $$
\end{remark}

As a consequence, the operator $\mu \mapsto \tr(JBJ^\#(-D^\# D^\# \mu+\mu E))$ is elliptic. This makes it possible to apply the maximum principle and obtain the infinitesimal rigidity statement needed.

\begin{corollary}\label{cor:max-princ}
  Under the assumptions of Lemma \ref{st:ads}, there is no non-zero solution of \eqref{eq:b6}.
\end{corollary}

\begin{proof}
  Let $\mu$ be a solution of \eqref{eq:b6}. Since $\mu$ is a smooth function on a closed surface $\Sigma$, $\mu$ achieves the maximum (resp. minimum) on $\Sigma$, say at $x_0$ (resp. $y_0$), and thus has a local maximum (resp. minimum) at $x_0$ (resp. $y_0$). By Lemma \ref{lm:neg-def}, at the local maximum $x_0$ of $\mu$, $\tr(JBJ^\#(-D^\# D^\#\mu))\leq 0$. Combined with \eqref{eq:b6}, we have $\tr(JBJ^\#(\mu E))\geq 0$ at $x_0$. Using Lemma \ref{lm:neg-def} again, it follows that $\mu(x_0)\leq 0$ and so $\mu\leq 0$ everywhere.

   On the other hand, at the local minimum $y_0$ of $\mu$, $\tr(JBJ^\#(-D^\# D^\#\mu))\geq 0$. Similarly as above, we can show that $\mu(y_0)\geq 0$ and so $\mu\geq 0$ everywhere. So $\mu=0$.
\end{proof}

\begin{proof}[Proof of Lemma \ref{st:ads}]
Combining Corollary \ref{cor:max-princ}, Definition \ref{def:b} and \eqref{eq:b-new}, we have $J\dot B=(E+JB)b=0.$  Lemma \ref{st:ads} follows.
\end{proof}

Before showing Proposition \ref{prop:local-rigidity}, we introduce the following proposition which ensures the existence and the uniqueness (up to isometries) of the maximal extension of a GHC AdS spacetime (see \cite[Theorem 3]{Choquet-Bruhat1969}).

\begin{proposition}
  \label{uniqueness of maximal extension}
  Let $(M,g)$ be a GHC AdS spacetime. There exists a unique (considered up to isometries) MGHC AdS spacetime $(M',g')$ with particles, called the maximal extension of $(M,g)$, in which $(M,g)$ can be isometrically embedded.
\end{proposition}

We are now ready to prove Proposition \ref{prop:local-rigidity}.

\begin{proof}[Proof of Proposition \ref{prop:local-rigidity}]
  Let $\dot{g}$ be a first-order deformation of $(M,g)\in \QF$ which preserves the left metric $I^\#_+$ and the induced metric $I$ on the strongly past-convex spacelike surface $\Sigma$ at first order. Assume that $\dot{g}$ is given by a one-parameter family $(M,g_t)_{t\in [0,\epsilon]}$ of quasifuchsian AdS spacetimes in $\QF$. The induced left metrics are denoted by $I^\#_{+,t}$, and the induced embedding data of $\Sigma$ are denoted by $(I_t,B_t)$. Up to a diffeomorphism of $M$ isotopic to the identity, we can assume that the induced metric $I_t$ on $\Sigma$ is fixed at first order (i.e. $\dot{I}=0$) and the first-order variation of $I^\#_+$ is induced by a vector field say $v$ on $\Sigma$ (which means that $\dot{I}^\#_+=\cL_vI^\#_+$). Then the assumptions of Lemma \ref{st:ads} are satisfied and we have $\dot{B}=0$.

To show that $\dot{g}=0$, we first claim that the embedding data $(\Sigma,I_t,B_t)$ uniquely determines a quasifuchsian AdS metric (which is $g_t$). Indeed, consider the manifold $\Sigma\times (-\frac{\pi}{2},0]$ with the following metric:
        \begin{equation*}
            h_t=-ds^{2}+I_t((\cos(s)E+\sin(s)B_t)\cdot,(\cos(s)E+\sin(s)B_t)\cdot)~,
        \end{equation*}
        where $E$ is the identity isomorphism on $T\Sigma$ and $s\in(-\frac{\pi}{2},0]$. Note that for each $s\in (-\frac{\pi}{2},0]$, the surface $\Sigma\times \{s\}$ is the equidistant surface at distance $s$ from the surface $\Sigma\times\{0\}$ on the convex side. The Lorentzian metric $h_t$ is a GHC AdS metric on $\Sigma\times(-\frac{\pi}{2},0]$.

        By Proposition \ref{uniqueness of maximal extension}, there exists a unique maximal extension of the AdS spacetime $(\Sigma\times(-\frac{\pi}{2},0], h_t)$, which is a MGHC AdS spacetime such that the restriction of its metric to the $\Sigma\times (-\frac{\pi}{2},0]$ (identified with a subset of $M$) is exactly $h_t$. Moreover, this maximal extension is $(M,g_t)$. Note that $\dot{I}=\dot{B}=0$. Taking the derivative of $h_t$ at $t=0$ shows that $\dot{h}=0$ and $\dot{g}$ is therefore zero when restricted to the subset $\Sigma\times (-\frac{\pi}{2},0]$. By the uniqueness of the maximal extension, $\dot{g}$ is therefore zero on the whole $M$. The proposition follows.
\end{proof}

\subsection{Proof of Theorem \ref{thm:closed-surf}}

To show Theorem \ref{thm:closed-surf}, we first construct the following map.

Let $h$ be a complete Riemannian metric of curvature $K<-1$ on $S$ and let $\cV(h)\subset\QF$ denote the space of quasifuchsian AdS spacetimes which contains a past-convex spacelike surface with induced metric isotopic to $h$.

\begin{definition}
Let $\Phi_h:\cV(h)\rightarrow \cT_S$ be the map which sends a manifold $M\in\cV(h)$ to its left metric $h_L$, identified with the metric $I^\#_+$ on $S$.
\end{definition}

We aim to show that $\Phi_h$ is a homeomorphism. Note that $\cV(h)$ can be alternatively interpreted as the space of isometric equivariant embeddings of $(S,h)$ into $\AdS^3$ (denoted by $\cI(S,h)$, see e.g. \cite[Section 3]{tamburelli2016}), which are given by a couple $(f,\rho)$, where $f:(\tilde{S},\tilde{h})\rightarrow \AdS^3$ is an isometric embedding into $\AdS^3$ with $(\tilde{S},\tilde{h})$ a universal Riemannian cover of $(S,h)$, and $\rho:\pi_1(S)\rightarrow {\rm PSL(2,\R)\times{\rm PSL}(2,\R)}$ is the representation such that $f(\gamma x)=\rho(\gamma)\circ f(x)$ for all $\gamma\in\pi_1(S)$ and $x\in\tilde{S}$.

It was shown by Tamburelli \cite[Lemma 3.2]{tamburelli2016} that $\cI(S,h)$, and thus the space $\cV(h)$ here, is a manifold of dimension $6\mathfrak{g}-6$, where $\mathfrak{g}$ is the genus of $S$, by identifying $\cI(S,h)$ with the space of solutions of Gauss-Codazzi equations and using classical techniques of elliptic operators. Moreover, he showed that $\Phi_h$ is proper \cite[Corollary 7.4]{tamburelli2016}.

\begin{proof}[Proof of Theorem \ref{thm:closed-surf}]
  Note that $\cV(h)$ has (real) dimension $6\mathfrak{g}-6$, which is the same as that of $\cT_S$. Combined with \ref{prop:local-rigidity}, the differential $d\Phi_h$ is an isomorphism. So $\Phi_h$ is a local homeomorphism. Combined with the fact that $\Phi_h$ is proper \cite[Corollary 7.4]{tamburelli2016}, $\Phi_h$ is a covering map of degree say $d(h)$. It remains to show that $d(h)=1$.

  Let $h_0$ be a metric with curvature $K<-1$ on $S$, and let $h_1$ be another metric on $S$, of constant curvature $K<-1$. As already mentioned, it follows from \cite[Theorem 1.15]{cyclic} that $\Phi_{h_1}:\cV(h_1)\to \cT_S$  is one-to-one. Let $(h_t)_{t\in[0,1]}$ be a smooth one-parameter family of Riemannian metrics on $S$ with curvature less than $-1$ (considered up to isotopy) connecting $h_0$ to $h_1$.

  Let
  $$ \cV = \{ (t,B_t)~|~ t\in [0,1], B_t\in \cV(h_t)\}~, $$
  and let $\Phi:\cV\to [0,1]\times \cT_S$ be the function defined by
  $$ \Phi(t,B_t) = (t, \Phi_{h_t}(B_t))~. $$
  We claim that at each point $(t,B_t)\in \cV$, $d\Phi$ is an isomorphism. Indeed, we know that $d\Phi_{|\{ 0\}\times T_{B_t}\cV(h_t)}$ is an isomorphisms from $\{ 0\}\times T_{B_t}\cV(h_t)$ to $\{ 0\}\times T_{\Phi_{h_t}(B_t)}\cT_S$ -- this is equivalent to the fact that $d\Phi_{h_t}$ is an isomorphism, as seen above. Moreover, the projection on the first factor in $\R\times T_{h_t}\cT_S$ of the image under $d\Phi$ of any vector of the form $(1, \dot B_t)$ is $1$.
   It follows that, in a suitable basis, $d\Phi$ at $(t,B_t)$ is block triangular, with one $1\times 1$ block with entry $1$, and one $(6\mathfrak{g}-6)\times (6\mathfrak{g}-6)$ which is invertible. So $d\Phi$ at $(t, B_t)$ is an isomorphism.

  Since each $\Phi_{h_t}, t\in [0,1]$ is proper, the function $\Phi$ is also proper. It is therefore a covering map. It follows that $d(h_0)=d(h_1)=1$. This applies to any metric $h_0$ on $S$ with curvature $K<-1$, and concludes the proof.
\end{proof}

\section{Application}

In this subsection, we apply Theorem \ref{thm:closed-surf} to give an alternative parametrization of the space $\mathcal{QF}_{AdS}(S)$.

It is known that the deformation space $\mathcal{QF}_{AdS}(S)$ of quasifuchsian AdS spacetimes can be parameterized in several ways, such as the Mess parameterization \cite{mess,bonsante-seppi:anti} by $\cT_S\times \cT_S$ in terms of the left and right metrics (which is closely related to Thurston's Earthquake Theorem for hyperbolic metrics on $S$ \cite{kerckhoff}), and the parametrization by $\cT_S\times \mathcal{ML}_S$ in terms of the induced metric and the bending lamination of the past (or future) boundary of the convex core \cite[Proposition 5.8]{cone}.

Furthermore, $\QF$ can also be parameterized by the cotangent bundle $T^{*}\cT_S$ of $\cT_S$, in terms of maximal spacelike surfaces in germs of AdS spacetimes (see \cite[Lemma 3.3]{minsurf}), or by $\cT_S\times \cT_S$ in terms of two hyperbolic metrics homothetic respectively to the first and third fundamental forms of past-convex constant Gaussian curvature $K$-surfaces, where $K\in(-\infty,-1)$, see \cite[Lemma 1.9]{cyclic}. Recently, Mazzoli and Viaggi have provided shear-bend coordinates for $\QF$ \cite{mazzoli-viaggi}.

Using the result of Theorem \ref{thm:closed-surf}, we give an alternative parametrization of $\QF$ below, which can be viewed as a mixed version of Mess parametrization and the parametrization in terms of the first and third fundamental forms of constant Gaussian curvature surfaces. Since the left and right metrics (resp. past-convex and future-convex spacelike surfaces, or the first and third fundamental forms) play symmetric roles in the parametrization, we only state the parametrization of $\QF$ here in terms of its left metric and the induced metric on the past-convex constant Gaussian curvature $K$-surfaces.

\begin{definition}
Let $K\in(-\infty,-1)$. We define a map $\phi_K:\QF\rightarrow\cT_S\times\cT_S$ by taking a quasifuchsian AdS spacetime $M\in\QF$ to its left metric and the hyperbolic metric homothetic to the induced metric on an embedded past-convex spacelike surface in $M$ of constant curvature $K$.
\end{definition}

The map $\phi_K$ is well-defined, by the fact given by Barbot, B\'{e}guin and Zeghib that each connected component of the complement of the convex core $C_M$ in a quasifuchsian AdS spacetime $M$ admits a unique foliation by spacelike surfaces of constant curvature $K$, with $K$ monotonic along the foliation, varying from $-1$ near the past (resp. future) boundary component of $C_M$ and $-\infty$ near the initial (resp.  final) singularity of $M$, see \cite{BBZ2}.

\begin{proposition}
For each $K\in(-\infty,-1)$, the map $\phi_K$ is a homeomorphism.
\end{proposition}

\begin{proof}
Fix $K\in(-\infty,-1)$. For any couple $(h,h')\in\cT_S\times\cT_S$, by applying the proof of Theorem \ref{thm:closed-surf} to the manifold $\cV((1/|K|)h')$, there is a unique $M\in\QF$ which has left metric isotopic to $h$ and contains a past-convex spacelike surface with induced metric isotopic to $(1/|K|)h'$. So the map $\phi_K$ is a bijection.

As the projection to the first factor of the holonomy representation of $M$, the left metric depends continuously on the geometric structures of $M$. Note that the induced metric on a past-convex spacelike surface of constant curvature $K$ depends continuously on the AdS metric of $M$, as can be seen from the relation between the induced metrics on past-convex surfaces of constant curvature
%%\marginnote{QY: added, the relation refers to the landslide map? is it known that the landslide $L_{e^{i\theta}}:\cT_S\times \cT_S\rightarrow \cT_S\times \cT_S$ defined in \cite{cyclic} is a homomorphism? If so, the proposition follows directly, see also the marginnote below.}
in quasifuchsian AdS spacetimes and their left and right metrics, see \cite[section 1.7]{cyclic}.
Moreover, the space of isometric equivariant embeddings of $(S,(1/|K|)h')$ into $\AdS^3$ is identified with the space of solutions of Gauss-Codazzi equations, which depends smoothly on the initial data $h'\in\cT_S$. As a consequence, the maps $\phi_K$ and $\phi^{-1}_K$ are both continuous. Combined with the above bijection of $\phi_K$, it follows that $\phi_K$ is a homeomorphism.
\end{proof}

% \begin{remark}
% For $K=-1$, the map $\phi_K$ sends $M\in\QF$ to its left metric $h_L$ and the induced metric $h_+$ on the future boundary component $\partial_+C_M$ (which is past-convex) of its convex core $C_M$. It is known that the left metric $h_L$ (resp. right metric $h_R$) can be obtained from $h_+$ by the left (resp. right) earthquake along the bending lamination $\lambda_+$ on $\partial_+C_M$:
%  $$h_L=E^L_{\lambda_+}(h_+)~,\quad \quad h_r=E^R_{\lambda_+}(h_+)~.$$
%  Using the parametrization of $\QF$ in terms of induced metric and bending lamination on $\partial_+C_M$, one can see that the parametrization of $\QF$ given by the map $\phi_K$ is closely related to Thurston's Earthquake Theorem for hyperbolic metrics on $S$. So
%  It follows from this line or reasoning that $\phi_K$ is also a homeomorphism for $K=-1$. \marginnote{JM: I find this paragraph not to be very clear, it would be better to have a proof. However this result is probably well-known so we could just remove this remark. QY: we could remove this remark if it's well-known. Besides, the above map $\phi_K$ actually closely relates to the landslide theorem (for instance, the bijection of $\phi_K$ should follow from the landslide theorem), should we mention this?}
% \end{remark}

\bibliographystyle{alpha}
\bibliography{/home/jean-marc/Dropbox/papiers/outils/biblio}
\end{document}